\documentclass[12pt,draftclsnofoot,onecolumn]{IEEEtran}
\usepackage{amsmath,graphicx,bm,amssymb,amsthm,enumerate,epsfig,psfrag,cases,mathtools}
\usepackage[shortlabels]{enumitem}
\renewcommand{\(}{\left(}
\renewcommand{\)}{\right)}
\renewcommand{\[}{\left[}
\renewcommand{\]}{\right]}

\newcommand{\End}[1]{{\rm{End}}}

\newtheorem{lemma}{Lemma}

\newtheorem{theorem}{Theorem}

\usepackage{fontenc}
\usepackage{inputenc}
\usepackage[square,sort,compress,comma,numbers]{natbib}
\usepackage{epstopdf,pst-node}

\usepackage{mathtools}

\begin{document}
\title{Central Limit Theorem for \\ Symmetric Exchangeable Random Variables}
\author{
    \IEEEauthorblockN{Ilya Soloveychik} \\
    \IEEEauthorblockA{\normalsize The Hebrew University of Jerusalem}
}

\maketitle

\vspace{-0.8cm}
\begin{abstract}
A central limit theorem for arrays of symmetric row-wise exchangeable random variables is presented. The result is valid for finite and infinite extendable and non-extendable sequences. Unlike most reported versions of the central limit theorem valid only for partial sums of finite non-extendable sequences, ours applies to the entire sum and mimics the classical theorem in this sense.
\end{abstract}

\section{Introduction}
Given a finite or infinite sequence $\{X_i\}$ of random variables, we say that it is exchangeable if the joint distribution of any finite subset of variables is invariant under arbitrary permutations of the variables. Exchangeable random variables were first introduced by de Finetti \cite{de1929funzione, de1937prevision} as a direct and natural generalization of independent and identically distributed  sequences. Unlike the case of independent variables, the behavior of sums of exchangeable variables is significantly harder to study. De Finetti \cite{de1929funzione, de1937prevision} and some of his followers focused on infinite exchangeable sequences. There exist, however, finite sets of exchangeable random variables which cannot be embedded into infinite exchangeable sequences, these are called \textit{finitely exchangeable} or \textit{non-extendable}, see \cite{konstantopoulos2019extendibility} and references therein. The analysis of extendable sequences can be reduced to the analysis of infinite sequences \cite{taylor1988limit}. On the other hand, the non-extendable sequences require quite different approaches. Most known central limit theorems for exchangeable non-extendable random sequences apply only to partial sums under additional conditions on the behavior of the entire sum \cite{weber1980martingale, chernoff1958central}. In this letter, we show that when the variables have symmetric marginal distributions, the classical central limit theorem applies to them.

\section{Central Limit Theorems for Exchangeable Arrays}
Let $\{k_n\}_{n=1}^\infty$ and $\{m_n\}_{n=1}^\infty$ be two sequences of natural numbers such that $k_n < m_n$ and
\begin{equation}
\label{eq:m_n_cond}
\frac{k_n}{m_n} \to \gamma \in [0,1).
\end{equation}
When dealing with finite exchangeable sequences, we consider triangular arrays of row-wise exchangeable variables of the form $\{X_{ni}\}_{n,i=1}^{\infty,m_n}$.

\begin{lemma}[Central Limit Theorem for Exchangeable Arrays, Theorem 2 from \cite{weber1980martingale}]
\label{eq:martin_theorem}
Let $\{X_{ni}\}_{n,i=1}^{\infty,m_n}$ be a triangular array of row-wise exchangeable random variables such that
\begin{enumerate}
\item $\mathbb{E}\[X_{n1}X_{n2}\] \to 0,\;\; n \to \infty$,
\item $\max\limits_{1\leqslant i \leqslant m_n} \frac{|X_{ni}|}{\sqrt{m_n}} \xrightarrow{P} 0,\;\; \forall n$,
\item $\frac{1}{k_n}\sum_{i=1}^{k_n} X_{ni}^2 \xrightarrow{P} 1,\;\; n \to \infty$.
\end{enumerate}
Then
\begin{equation}
\sqrt{k_n}\[\frac{1}{k_n}\sum_{i=1}^{k_n}X_{ni} - \frac{1}{m_n}\sum_{i=1}^{m_n}X_{ni}\] \xrightarrow{L} \mathcal{N}(0,1-\gamma),\;\; n \to \infty.
\end{equation}
\end{lemma}

The main result of this letter reads as follows.
\begin{theorem}[Central Limit Theorem for Symmetric Exchangeable Arrays]
\label{thm:main}
Let $\{X_{ni}\}_{n,i=1}^{\infty,m_n}$ be a triangular array of symmetric row-wise exchangeable random variables such that
\begin{enumerate}
\item $\mathbb{E}\[X_{n1}X_{n2}\] \to 0,\;\; n \to \infty$,
\item $\max\limits_{1\leqslant i \leqslant m_n} \frac{|X_{ni}|}{\sqrt{m_n}} \xrightarrow{P} 0,\;\; \forall n$,
\item $\frac{1}{m_n}\sum_{i=1}^{m_n} X_{ni}^2 \xrightarrow{P} 1,\;\; n \to \infty$.
\end{enumerate}
Then
\begin{equation}
\frac{1}{\sqrt{m_n}}\sum_{i=1}^{m_n}X_{ni} \xrightarrow{L} \mathcal{N}(0,1),\;\; n \to \infty.
\end{equation}
\end{theorem}

\begin{proof}
Assume without loss of generality that $m_n$ is an even number and set
\begin{equation}
k_n = \frac{m_n}{2}.
\end{equation}
For every $n$, consider the following sequence,
\begin{equation}
Y_{ni} = \begin{cases} X_{ni}, & i\leqslant k_n, \\ -X_{ni}, & i> k_n. \end{cases}
\end{equation}
Due to the symmetry of the marginal distributions, the new sequence is exchangeable with the same joint distribution as the original sequence. Note that requirement 3) of Lemma \ref{eq:martin_theorem} is naturally fulfilled due to the exchangeability. The other conditions of Lemma \ref{eq:martin_theorem} are clearly satisfied as well and we obtain,
\begin{equation}
\sqrt{\frac{m_n}{2}}\[\frac{2}{m_n}\sum_{i=1}^{m_n/2}Y_{ni}-\frac{1}{m_n}\sum_{i=1}^{m_n}Y_{ni}\] \xrightarrow{L} \mathcal{N}\(0,\frac{1}{2}\),\;\; n \to \infty.
\end{equation}
Consider the following sequence of equalities,
\begin{multline}
\frac{2}{m_n}\sum_{i=1}^{m_n/2}Y_{ni}-\frac{1}{m_n}\sum_{i=1}^{m_n}Y_{ni} = \frac{2}{m_n}\sum_{i=1}^{m_n/2}X_{ni}-\frac{1}{m_n}\[\sum_{i=1}^{m_n/2}X_{ni} - \sum_{i=m_n/2+1}^{m_n}X_{ni}\] \\ = \frac{1}{m_n}\sum_{i=1}^{m_n/2}X_{ni} + \frac{1}{m_n} \sum_{i=m_n/2+1}^{n}X_{ni} = \frac{1}{m_n}\sum_{i=1}^{m_n}X_{ni},
\end{multline}
to get
\begin{equation}
\sqrt{\frac{m_n}{2}}\cdot\frac{1}{m_n}\sum_{i=1}^{m_n}X_{ni} \xrightarrow{L} \mathcal{N}\(0,\frac{1}{2}\),\;\; n \to \infty,
\end{equation}
or
\begin{equation}
\frac{1}{\sqrt{m_n}}\sum_{i=1}^{m_n}X_{ni} \xrightarrow{L} \mathcal{N}\(0,1\),\;\; n \to \infty.
\end{equation}
Note that if $m_n$ is odd we just disregard the last element of the sum which does not change the limiting distribution since the effect of one variable is limited by the assumptions. This concludes the proof.
\end{proof}

\bibliography{}

\end{document}